\documentclass{article}
\usepackage[utf8]{inputenc}
\usepackage[a4paper,top=2.5cm,bottom=3cm,left=2.5cm,right=2.5cm,marginparwidth=2cm]{geometry}
\usepackage[hidelinks]{hyperref}
\usepackage{amsmath, amssymb, amsthm, bbm, graphicx, cleveref, caption, tikz, subcaption}
\graphicspath{ {figures/} }

\usepackage{todonotes}

\usepackage{fancyhdr}
\usepackage{lastpage}
\usepackage{titlefoot}
\usepackage{titling}
\thanksmarkseries{arabic}
\newcommand\extrafootertext[1]{%
    \bgroup
    \renewcommand\thefootnote{\fnsymbol{footnote}}%
    \renewcommand\thempfootnote{\fnsymbol{mpfootnote}}%
    \footnotetext[0]{#1}%
    \egroup
}

\newtheorem{thm}{Theorem}

\newtheorem{lemma}[thm]{Lemma}
\newtheorem{defn}[thm]{Definition}
\newtheorem{cor}[thm]{Corollary}
\theoremstyle{definition}
\newtheorem{remark}[thm]{Remark}

\numberwithin{equation}{section}
\numberwithin{thm}{section}

\newcommand{\ddp}[2]{\frac{\partial#1}{\partial#2}}

\newcommand{\D}{\partial D}
\renewcommand{\S}{\mathcal{S}}
\newcommand{\outside}{\mathbb{R}^3\setminus \overline{D}}
\newcommand{\uin}{u^{\mathrm{in}}}
\newcommand{\C}{\mathcal{C}}
\newcommand{\de}{\,\mathrm{d}}
\newcommand{\iu}{\mathrm{i}}
\newcommand{\vn}{\textit{\textbf{v}}_n}
\newcommand{\vj}{\textit{\textbf{v}}_j}
\newcommand{\vk}{\textit{\textbf{v}}_k}
\newcommand{\omold}{\omega^{\text{old}}}
\newcommand{\omnew}{\omega^{\text{new}}}

\title{Robustness of subwavelength devices: a case study of cochlea-inspired rainbow sensors}
\author{Bryn Davies \thanks{Department of Mathematics, Imperial College London, 180 Queen's Gate, London SW7 2AZ, United Kingdom} \and Laura Herren \thanks{Department of Statistics and Data Science, Yale University, New Haven, CT 06511, USA} }
\date{}

\begin{document}

\maketitle
\extrafootertext{Correspondence should be addressed to: bryn.davies@imperial.ac.uk}

\begin{abstract}
    The aim of this work is to derive precise formulas which describe how the properties of subwavelength devices are changed by the introduction of errors and imperfections. As a demonstrative example, we study a class of cochlea-inspired rainbow sensors. These are devices based on a graded array of subwavelength resonators which have been designed to mimic the frequency separation performed by the cochlea. We show that the device's properties (including its role as a signal filtering device) are stable with respect to small imperfections in the positions and sizes of the resonators. Additionally, if the number of resonators is sufficiently large, then the device's properties are stable under the removal of a resonator.
\end{abstract}

\vspace{0.5cm}
		
\vspace{0.2cm}
		

\vspace{0.5cm}

\section{Introduction}

The cochlea is the key organ of mammalian hearing, which filters sounds according to frequency and then converts this information to neural signals. Across the biological world, including in humans, cochleae have remarkable abilities to filter sounds at a very high resolution, over a wide range of volumes and frequencies. This exceptional performance has given rise to a community of researchers seeking to design artificial structures which mimic the function of the cochlea \cite{davies2019fully, davies2020hopf, babbs2011quantitative, karlos2020cochlea, rupin2019mimicking, zhao2019metamaterial}. These devices are based on the phenomenon known as \textit{rainbow trapping}, whereby frequencies are separated in graded resonant media. This has been observed in a range of settings including acoustics \cite{zhu2013acoustic}, optics \cite{tsakmakidis2007trapped} (where the term `rainbow trapping' was coined), water waves \cite{bennetts2018graded} and plasmonics \cite{jang2011plasmonic}, among others.

The motivation for designing cochlea-inspired sensors is twofold. Firstly, it is hoped that they can be used to design artificial hearing approaches, either through the realisation of physical devices \cite{rupin2019mimicking, joyce2015developing} or by informing computational algorithms \cite{ammari2020biomimetic, lyon2017human}. Additionally, it is hoped that modelling and building these devices will yield new insight into the function of the cochlea itself. The cochlea is a small organ that is buried inside an organism's head, meaning that experiments on living samples is exceptionally difficult. This means that many of the characteristics which are unique to living specimens are still poorly understood. The nature of the amplification mechanism used by the cochlea is a prime example of this \cite{hudspeth2008making}. It is hoped that studying artificial cochlea-inspired devices, which can be both modelled and experimented on more easily, will yield new clues into the possible forms of this amplification \cite{davies2020hopf, rupin2019mimicking, joyce2015developing}.

Micro-structured media with strongly dispersive behaviour, such as the cochlea-like rainbow sensors considered here, are examples of \textit{acoustic metamaterials}. Metamaterials are a diverse collection of materials that have extraordinary and `unnatural' properties, such as negative refractive indices and the ability to support cloaking effects \cite{craster2012acoustic, kadic2019review}. One of the challenges in this field, however, is that errors and imperfections are inevitably introduced when devices are manufactured, which has the potential to significantly alter their function. For this reason, a large field has emerged studying \textit{topologically protected} structures, whose properties experience greatly enhanced robustness thanks to the topological properties of the underlying periodic media \cite{khanikaev2013photonic, fefferman2017topologically, ammari2020topologically}. While the theory of topopogical protection has deep implications for the design of rainbow sensors \cite{chaplain2020delineating, chaplain2020topological}, there is yet to be an established link with biological structures and we will study a conventional graded metamaterial in this work. 

The aim of this work is to derive formulas which describe how the properties of a cochlea-inspired rainbow sensor are affected by the introduction of errors and imperfections. This will give quantitative insight into the extent to which these devices are robust with respect to manufacturing errors. It may also yield insight into the cochlea itself, which has a remarkable ability to function effectively even when significantly damaged. As depicted in \Cref{fig:cochlea_damage}, cochlear receptor cells are often significantly damaged in older organisms. However, it has been observed that humans can lose as much as 30--50\% of their receptor cells without any perceptible loss of hearing function \cite{CDC, wu2020age}. This remarkable robustness is part of the motivation for this study: how do cochlea-inspired rainbow sensors behave under similar perturbations? 

We will study a passive device consisting of an array of material inclusions whose properties resemble those of air bubbles in water. These inclusions act as resonators, oscillating with so-called breathing modes, and exhibit resonance at subwavelength scales, often known as \textit{Minnaert resonance} \cite{minnaert1933musical, devaud2008minnaert, ammari2018minnaert}. Devices have been built based on these principles by injecting bubbles into polymer gels \cite{leroy2009design, leroy2009transmission}. It was shown in \cite{davies2019fully} that by grading the size of the resonators, to give the geometry depicted in \Cref{fig:rainbow}, it is possible to replicate the spatial frequency separation of the cochlea.

We will use boundary integral methods to analyse the scattering of the acoustic field by the cochlea-inspired rainbow sensor \cite{ammari2018mathematical}. We will define the notion of subwavelength resonance as an asymptotic property, in terms of the material contrast, and perform an asymptotic analysis of the structure's resonant modes. This first-principles approach yields an approximation in terms of the \textit{generalized capacitance matrix}. We will recap  this theory in \Cref{sec:prelims} and refer the reader to \cite{capacitance} for a more thorough exposition. In \Cref{sec:imperfections}, we study the effect of small perturbations to the size and position of the resonators. The derived formulas show that the rainbow sensor's properties are stable with respect to these imperfections. Then, in \Cref{sec:remove}, we examine more drastic perturbations, namely those caused by removing resonators from the array. This is inspired by the images in \Cref{fig:cochlea_damage}, where in many places the receptor cell stereocilia have been completely destroyed. We will show that, provided that array is sufficiently large, the sensor's properties are nonetheless stable. Finally, in \Cref{sec:signals}, we study the equivalent signal transformation that is induced by the cochlea-inspired rainbow sensor and show that its properties are stable with respect to changes in the device.

\begin{figure}
    \centering
    \includegraphics[trim=0 8.3cm 0 0, clip,width=0.45\linewidth]{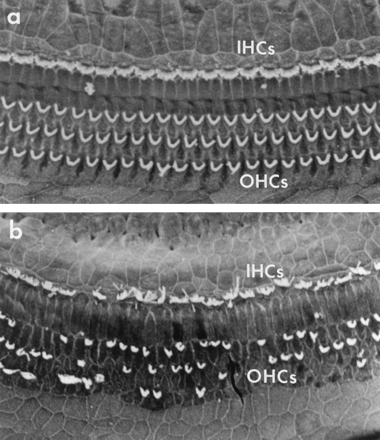}
    \hspace{0.2cm}
    \includegraphics[trim=0 0.8cm 0 7.5cm, clip,width=0.45\linewidth]{HCdamage.png}
    \caption{The receptor cells in a (a) normal and (b) damaged cochlea. The receptor cells are arranged as one row of inner hair cells (IHCs) and three rows of outer hair cells (OHCs). In a damaged cochlea, the stereocilia are severely deformed and, in many cases, missing completely. The images are scanning electron micrographs of rat cochleae, provided by Elizabeth M Keithley.}
    \label{fig:cochlea_damage}
\end{figure}

\begin{figure}
    \begin{center}
        \includegraphics[width=0.8\linewidth]{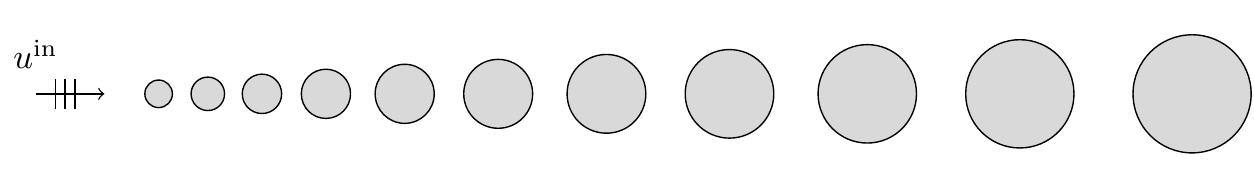}
    \end{center}
    \caption{A cochlea-inspired rainbow sensor. The gradient in the sizes of the resonators means the device separates different frequencies in space: higher frequencies will give a peak amplitude to the left of the array, while lower frequencies will give a maximal response further to the right. This mimics the action of the cochlea in filtering sound waves.} \label{fig:rainbow}
\end{figure}

\section{Mathematical preliminaries} \label{sec:prelims}
\subsection{Problem setting}

Will will study a Helmholtz scattering problem to model the scattering of time-harmonic acoustic waves by the resonator array. The resonators are modelled as material inclusions $D_1,\dots,D_N$ which are disjoint, bounded and have boundaries in $C^{1,\alpha}$ for some $0<\alpha<1$. We denote the wave speeds inside the resonators as $v$ and in the background medium as $v_0$. For an angular frequency $\omega$ we introduce the wavenumbers
\begin{equation*}
    k=\frac{\omega}{v} \quad\text{and}\quad k_0=\frac{\omega}{v_0}.
\end{equation*}
Additionally, we introduce the dimensionless contrast parameter
\begin{equation}
    \delta=\frac{\rho}{\rho_0},
\end{equation}
which is the ratio of the densities of the materials inside and outside the resonators. The scattering problem, due to the resonator array 
\begin{equation}
D=D_1\cup\dots\cup D_N,    
\end{equation}
is then given by
\begin{equation} \label{eq:helmholtz_equation}
\begin{cases}
\left( \Delta + k_0^2 \right) u = 0 & \text{in } \outside, \\
\left( \Delta + k^2 \right) u = 0, & \text{in } D, \\
u_+ - u_- = 0, & \text{for } \D,\\
\delta \ddp{u}{\nu}\big|_+ -  \ddp{u}{\nu}\big|_- = 0, & \text{on } \D,\\
u^s := u - \uin \text{ satisfies the SRC}, & \text{as } |x|\to\infty,
\end{cases}
\end{equation}
where the SRC refers to the Sommerfeld radiation condition, which guarantees that the scattered waves radiate energy outwards to the far field \cite{ammari2018mathematical}.

\begin{defn}[Resonance]
We define a \emph{resonant frequency} to be $\omega\in\mathbb{C}$ such that there exists a non-zero solution $u$ to \eqref{eq:helmholtz_equation} in the case that $\uin=0$. The solution $u$ is the \emph{resonant mode} associated to $\omega$.
\end{defn}

In this work, we will characterise \emph{subwavelength} resonance in terms of the limit of the contrast parameter $\delta$ being small. In particular, we assume that
\begin{equation}
    \delta\ll1 \quad\text{while}\quad v,v_0,\tfrac{v}{v_0}=O(1) \text{ as } \delta\to0.
\end{equation}
This approach allows us to fix the size and position of the resonators and study subwavelength resonant modes as those which exist at asymptotically low frequencies when $\delta$ is small.
\begin{defn}[Subwavelength resonance] \label{defn:subw_res}
We define a subwavelength resonant frequency to be a resonant frequency $\omega=\omega(\delta)$ that depends continuously on $\delta$ and satisfies
\begin{equation*}
    \omega\to0 \quad\text{as}\quad \delta\to0.
\end{equation*}
\end{defn}
This asymptotic approach has been shown to be effective at modelling devices based on the canonical example of air bubbles in water \cite{ammari2018minnaert, capacitance}, where the contrast parameter is approximately $\delta\approx 10^{-3}$. Furthermore, this asymptotic definition of subwavelength resonance reveals that there is a fundamental difference between these resonant modes and those which are not subwavelength, and leads to the following existence result:

\begin{lemma}
    A system of $N$ subwavelength resonators has $N$ subwavelength resonant frequencies with positive real part.
\end{lemma}
\begin{proof}
This follows using Gohberg-Sigal theory to perturb the solutions that exist in the limiting case where $\delta=0$, $\omega=0$, see \cite{capacitance, ammari2018mathematical} for details.
\end{proof}

The subwavelength resonant frequencies of a cochlea-inspired rainbow sensor composed of 22 subwavelength resonators are shown in \Cref{fig:fullspectrum}. The multipole expansion method (see the appendices of \cite{ammari2020topologically} for details) is used to simulate an array of spherical resonators which is which 35mm long and has the material parameters of air bubbles in water. The real parts of the resonant frequencies span the range 7.4kHz--33.8kHz (\Cref{fig:fullspectrum} shows angular frequency). This range can be fine tuned to match the desired function (or to match the range of human hearing more closely) \cite{davies2020hopf}. The negative imaginary parts describe the loss of energy to the far field.

\begin{figure}
    \centering
    \includegraphics[width=0.65\linewidth]{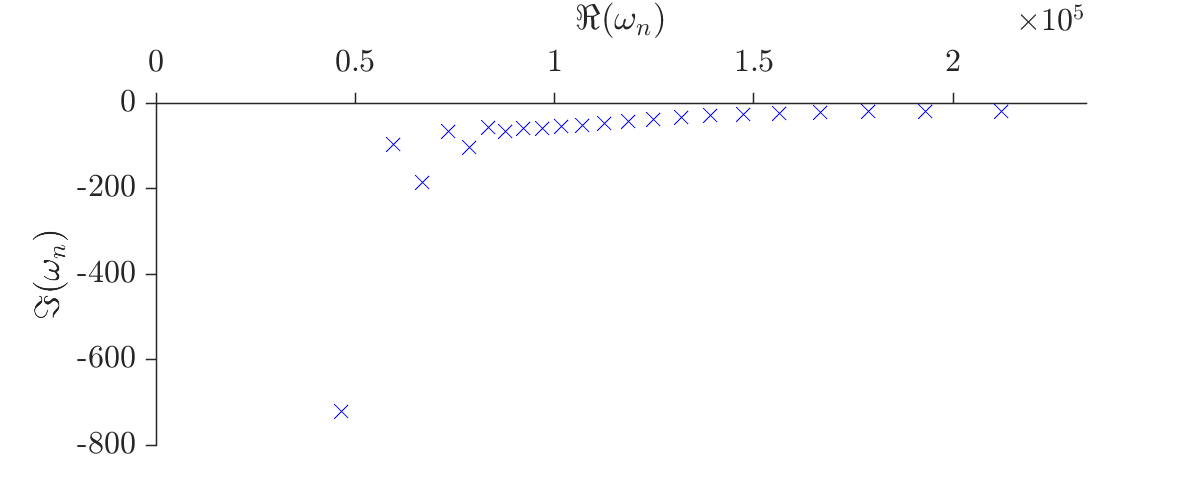}
    \caption{The 22 subwavelength resonant frequecies of a cochlea-inspired rainbow sensor composed of 22 subwavelength resonators, plotted in the lower-right complex plane.}
    \label{fig:fullspectrum}
\end{figure}

\subsection{Boundary integral operators}

In order to model the scattering of waves by the array $D$ we will use layer potentials to represent solutions. 

\begin{defn}[Single layer potential]
Given a bounded domain $D\subset\mathbb{R}^3$ and a wavenumber $k\in\mathbb{C}$ we define the Helmholtz single layer potential as
\begin{equation*}
    \S_D^k[\varphi](x)=\int_{\D} G^k(x-y)\varphi(y)\de\sigma(y), \qquad \varphi\in L^2(\D), \, x\in\mathbb{R}^3,
\end{equation*}
where the Green's function $G$ is given by
\begin{equation*}
    G^k(x)=-\frac{e^{\iu k|x|}}{4\pi|x|}, \qquad x\neq0.
\end{equation*}
\end{defn}

The value of the single layer potential is that we can use it to represent solutions to the Helmholtz scattering problem \eqref{eq:helmholtz_equation}. In particular, there exist some densities $\psi,\phi\in L^2(\D)$ such that
\begin{equation} \label{eq:representation}
    u(x)=\begin{cases}
    \uin(x)+\S_D^{k_0}[\psi], & x\in\outside,\\
    \S_D^k[\phi], & x\in D.
    \end{cases}
\end{equation}
This representation means that the Helmholtz equations and the radiation condition from \eqref{eq:helmholtz_equation} are necessarily satisfied. It remains only to find densities $\psi,\phi\in L^2(\D)$ such that the two transmission conditions across the boundary $\D$ are satisfied. See \cite{ammari2018mathematical} more details on the use of layer potentials in modelling scattering problems. In this work, we will make use of some elementary properties. Since we define subwavelength resonance as an asymptotic property (\Cref{defn:subw_res}), we will make use of the asymptotic expansion
\begin{equation}
    \S_D^k=\S_D^0+k\S_{D,1}+O(k^2), \qquad\text{as } k\to0,
\end{equation}
where $\S_{D,1}[\varphi]=(4\pi\iu)^{-1}\int_{\D}\varphi\de\sigma$ and convergence holds in the operator norm. In order to derive leading-order approximations, we will make use of the fact that $S_D^0$ is invertible \cite{ammari2018mathematical}:

\begin{lemma}
$S_D^0$ is invertible as a map from $L^2(\D)$ to $H^1(\D)$.
\end{lemma}

\subsection{The generalized capacitance matrix}

Studying the subwavelength resonant properties of the high-contrast structure as an asymptotic property in terms of $\delta\ll1$ leads to a concise characterisation of the resonant states. In particular, we find that the leading-order properties of the resonant frequencies and associated eigenmodes are given in terms of the eigenstates of the \emph{generalized capacitance matrix}, as introduced in \cite{capacitance}. This is a generalization of the notion of capacitance that is widely used in electrostatics to model the distributions of potential and charge in a system of conductors \cite{diaz2011positivity}.

\begin{defn}[Capacitance matrix] \label{def:cap}
Given $N\in\mathbb{N}$ disjoint inclusions $D_1,\dots,D_N\subset\mathbb{R}^3$, the associated capacitance matrix $C\in\mathbb{R}^{N\times N}$ is defined as
\begin{equation*}
    \C_{ij}=-\int_{\D_i}(\S_D^0)^{-1}[\chi_{\D_j}]\de\sigma,\qquad i,j=1,\dots,N,
\end{equation*}
where $\chi_{\D_i}$ is the characteristic function of the boundary $\D_i$.
\end{defn}

In this work, we are interested in cochlea-like rainbow sensors that have resonators with increasing size. In general, in order to use capacitance coefficients to understand the resonant properties of an array of non-identical resonators we need to re-scale the coefficients. The \emph{generalized capacitance matrix} that we obtain is studied at length in \cite{capacitance}. With this approach, we can study arrays of resonators with different sizes, shapes and material parameters. In this work, we are assuming the resonators all have the same interior material parameters (given by the wave speed $v$ and contrast parameter $\delta$) so only need to weight according to the different sizes of the resonators.

\begin{defn}[Volume scaling matrix]
    Given $N\in\mathbb{N}$ disjoint inclusions $D_1,\dots,D_N\subset\mathbb{R}^3$ the volume scaling matrix $V\in\mathbb{R}^{N\times N}$ is the diagonal matrix given by
    \begin{equation*}
        V_{ii}=\frac{1}{\sqrt{|D_i|}}, \qquad i=1,\dots,N,
    \end{equation*}
    where $|D_i|$ is the volume of $D_i$.
\end{defn}

\begin{defn}[Generalized capacitance matrix] \label{def:gencap}
Given $N\in\mathbb{N}$ disjoint inclusions $D_1,\dots,D_N\subset\mathbb{R}^3$ with identical interior material parameters, the associated (symmetric) generalized capacitance matrix $\C\in\mathbb{R}^{N\times N}$ is defined as
\begin{equation*}
    \C=VCV.
\end{equation*}
\end{defn}

In previous works, the generalized capacitance is more often defined as the asymmetric matrix $V^2C$ (see \cite{capacitance} and references therein). Here, we will want to use some of the many existing results about perturbations of eigenstates of symmetric matrices so opt for the symmetric version. Note that $\C=VCV$ is similar to $V^2C$. The value of of the generalized capacitance matrix is clear from the following results.

\begin{thm} \label{thm:res}
	Consider a system of $N$ subwavelength resonators in $\mathbb{R}^3$ and let $\{(\lambda_n,\vn): n=1,\dots,N\}$ be the eigenpairs of the (symmetric) generalized capacitance matrix $\mathcal{C}\in\mathbb{R}^{N\times N}$. As $\delta\to0$, the subwavelength resonant frequencies satisfy the asymptotic formula
	\begin{equation*}
	\omega_n = \sqrt{\delta v^2 \lambda_n}-\iu\delta\tau_n+O(\delta^{3/2}), \quad n=1,\dots,N,
	\end{equation*}
	where the second-order coefficients $\tau_n$ are given by
	\begin{equation*}
	    \tau_n=\frac{v^2}{8\pi v_0}\frac{1}{\|\textbf{v}_n\|^2}\textbf{v}_n^\top VCJCV\textbf{v}_n, \quad n=1,\dots,N,
	\end{equation*}
	with $J$ being the $N\times N$ matrix of ones.
\end{thm}

\begin{cor} \label{prop:eigenvector}
	Let $\textbf{v}_n$ be the normalized eigenvector of $\mathcal{C}$ associated to the eigenvalue $\lambda_n$. Then the normalized resonant mode $u_n$ associated to the resonant frequency $\omega_n$ is given, as $\delta\to0$, by
	\begin{equation*}
	u_n(x)=\begin{cases}
	\textbf{v}_n^\top V\textbf{S}_D^{k_0}(x)+O(\delta^{1/2}), \quad x\in\mathbb{R}^3\setminus \overline{D}, \\
	\textbf{v}_n^\top V\textbf{S}_D^{k}(x)+O(\delta^{1/2}), \quad x\in D,
	\end{cases}
	\end{equation*}
	where $\textbf{S}_D^k:\mathbb{R}^3\to\mathbb{C}^N$ is the vector-valued function given by
	\begin{equation*}
	\textbf{S}_D^k (x)=\begin{pmatrix}
	\mathcal{S}_D^k[\psi_1](x) \\[-0.4em]
	\vdots \\[-0.3em]
	\mathcal{S}_D^k[\psi_N](x)
	\end{pmatrix},  \quad x\in\mathbb{R}^3\setminus \partial D,
	\end{equation*}
	with $\psi_i:=(\mathcal{S}_D^0)^{-1}[\chi_{\partial D_i}]$.
\end{cor}

\begin{remark}
Since $C$ is symmetric, $V$ is diagonal and $J$ is positive semi-definite, it holds that $\tau_n\geq0$ for all $n=1,\dots,N$. This corresponds to the loss of energy from the system.
\end{remark}

\begin{remark}
We will shortly want to study how the properties of the generalized capacitance matrix $\C$ vary when changes are made to the structure $D$. For this reason, we will often write $\C=\C(D)$ to emphasise the dependence of the generalized capacitance matrix on the geometry of $D$. Similarly, we will write $\lambda_i=\lambda_i(D)$ and $\tau_i=\tau_i(D)$ for the quantities from \Cref{thm:res}.
\end{remark}

\section{Imperfections in the device} \label{sec:imperfections}

We will begin by deriving formulas to describe the effects of making small perturbations to the positions and sizes of the resonators, as depicted in \Cref{fig:size&pos}. Perturbations of this nature are important as they will be introduced when a device is manufactured. The results in this section give quantitative estimates on the extent to which the perturbations of the structure's properties are stable with respect to small imperfections.

\subsection{Dilute approximations} \label{sec:dilute}

In order to simplify the analysis, and to allow us to work with explicit formulas, we will make an assumption that the resonators are small compared to the distance between them. In particular, we will assume that each resonator $D_i$ is given by $B_i+\epsilon^{-1}z_i$ where $B_i\subset\mathbb{R}^3$ is some fixed domain, $z_i\in\mathbb{R}^3$ is some fixed vector and $0<\epsilon\ll1$ is some small parameter. We will assume that each fixed domain $B_i$, for $i=1,\dots,N$, is positioned so that it contains the origin and that the complete structure is given by
\begin{equation} \label{eq:dilute}
    D=\bigcup_{i=1}^N D_i, \qquad D_i= \left( B_i + \epsilon^{-1} z_i \right).
\end{equation}
Under this assumption, the generalized capacitance matrix has an explicit leading-order asymptotic expression in terms of the \emph{dilute} generalized capacitance matrix:

\begin{defn}[Dilute generalized capacitance matrix] \label{defn:diluteGCM}
Given $0<\epsilon\ll1$ and a resonator array that is $\epsilon$-dilute in the sense of \eqref{eq:dilute}, the associated dilute generalized capacitance matrix $\C^\epsilon\in\mathbb{R}^{N\times N}$ is defined as
\begin{equation*}
    \C_{ij}^\epsilon=
    \begin{cases}
\frac{\mathrm{Cap}_{B_{i}}}{|B_i|}, & i=j, \\
-\epsilon\frac{\mathrm{Cap}_{B_{i}} \mathrm{Cap}_{B_{j}}}{4\pi |z_{i}-z_{j}|\sqrt{|B_i||B_j|}}, & i\neq j,
\end{cases}
\end{equation*}
where we define the capacitance $\mathrm{Cap}_B$ of a set $B\subset\mathbb{R}^3$ as 
\begin{equation*}
    \mathrm{Cap}_{B_i}:=-\int_{\partial B} (\S_B^0)^{-1}[\chi_{\partial B}]\de\sigma.
\end{equation*}
\end{defn}

\begin{lemma}\label{diluteness}
Consider a resonator array that is $\epsilon$-dilute in the sense of \eqref{eq:dilute}. In the limit as $\epsilon \rightarrow 0$, the asymptotic behaviour of the (symmetric) generalized capacitance matrix is given by
\begin{equation*}
    \C=\C^\epsilon+O(\epsilon^2) \quad\text{as}\quad\epsilon\to0.
\end{equation*}
\end{lemma}
\begin{proof}
This was proved in \cite{ammari2021high} and is a modification of a result from \cite{ammari2020topologically}.
\end{proof}

\begin{remark}
It would also be possible to state an appropriate diluteness condition as a rescaling of the sizes of the resonators, by taking $D_i=\epsilon B_i+z_i$ in \eqref{eq:dilute}. This would give analogous results, as used in \cite{ammari2020topologically}.
\end{remark}

\subsection{Changes in size} \label{sec:size}

\begin{figure}
    \begin{center}
        \includegraphics[width=0.6\linewidth]{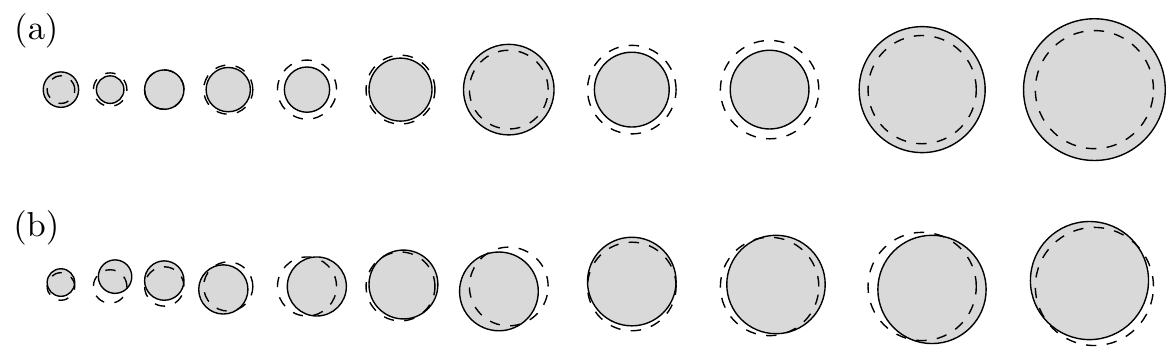}
    \end{center}
    \caption{We study the effects of adding random perturbations to the (a) size and (b) position of the resonators in a cochlea-inspired rainbow sensor. The original structure is shown in dashed.} \label{fig:size&pos}
\end{figure}

We first consider imperfections due to changes in the size of the resonators. In particular, suppose there exist some factors $\alpha_1,\dots,\alpha_N$ such that the perturbed structure is given by
\begin{equation} \label{eq:sizedef}
    D^{(\alpha)}=\bigcup_{i=1}^N \left((1+\alpha_i)B_i+\epsilon^{-1}z_i\right).
\end{equation}
We will assume that the perturbations $\alpha_1,\dots,\alpha_N$ are small in the sense that there exists some parameter $\alpha$ such that $\alpha_i=O(\alpha)$ as $\alpha\to0$.

\begin{lemma} \label{lem:sizeperturb}
Suppose that a resonator array $D$ is deformed to give $D^{(\alpha)}$, as defined in \eqref{eq:sizedef}, and that the size change parameters $\alpha_1,\dots,\alpha_N$ satisfy $\alpha_i=O(\alpha)$ as $\alpha\to0$ for all $i=1,\dots,N$. Then, the dilute generalized capacitance matrix associated to $D^{(\alpha)}$is given by
\begin{equation*}
    \C^\epsilon(D^{(\alpha)})=\C^\epsilon(D)+A(\alpha),
\end{equation*}
where $A(\alpha)$ is a symmetric ${N\times N}$-matrix whose Frobenius norm satisfies $\|A\|_F=O(\alpha)$ as $\alpha\to0$.
\end{lemma}
\begin{proof}
Making the substitution $B_i\mapsto (1+\alpha_i)B_i$ in \Cref{defn:diluteGCM} gives 
\begin{equation*}
    \C_{ij}^\epsilon(D^{(\alpha)})=
    \begin{cases}
 \frac{\mathrm{Cap}_{B_{i}}}{(1+\alpha_i)^2|B_i|}, & i=j, \\
-\epsilon\frac{\mathrm{Cap}_{B_{i}} \mathrm{Cap}_{B_{j}}}{4\pi |z_{i}-z_{j}|\sqrt{(1+\alpha_i)(1+\alpha_j)}\sqrt{|B_i||B_j|}}, & i\neq j.
\end{cases}
\end{equation*}
For small $\alpha$ we can expand the denominators (while keeping $\epsilon$ fixed) to give
\begin{equation*}
    \C_{ij}^\epsilon(D^{(\alpha)})=
    \begin{cases}
 (1-2\alpha_i)\frac{\mathrm{Cap}_{B_{i}}}{2|B_i|}+O(\alpha^2), & i=j, \\
-\epsilon\left(1-\frac{1}{2}(\alpha_i+\alpha_j)\right)\frac{\mathrm{Cap}_{B_{i}} \mathrm{Cap}_{B_{j}}}{4\pi |z_{i}-z_{j}|\sqrt{|B_i||B_j|}}+O(\alpha^2), & i\neq j,
\end{cases}
\end{equation*}
as $\alpha\to0$.
\end{proof}

\begin{thm} \label{thm:size}
    Suppose that a resonator array $D$ is $\epsilon$-dilute in the sense of \eqref{eq:dilute} and is deformed to give $D^{(\alpha)}$, as defined in \eqref{eq:sizedef}, for size change parameters $\alpha_1,\dots,\alpha_N$ which satisfy $\alpha_i=O(\alpha)$ as $\alpha\to0$ for all $i=1,\dots,N$. Then, the resonant frequencies satisfy
    \begin{equation*}
        |\omega_n(D)-\omega_n(D^{(\alpha)})|=O\left(\sqrt{\delta(\alpha+\epsilon^2)}\right).\qedhere
    \end{equation*}
    as $\alpha,\delta,\epsilon\to0$.
\end{thm}
\begin{proof}
From \Cref{lem:sizeperturb} we have that $\C^\epsilon(D^{(\alpha)})=\C^\epsilon(D)+A(\alpha)$ where $A$ is a symmetric $N\times N$-matrix. Then, by the Wielandt-Hoffman theorem \cite{golub2983matrix}, it holds that the eigenvalues of $\C^\epsilon(D)$ and $\C^\epsilon(D^{(\alpha)})$, which we denote by $\lambda_i^\epsilon(D)$ and $\lambda_i^\epsilon(D^{(\alpha)})$, respectively, satisfy
\begin{equation} \label{eq:WielandtHoffman}
    \sum_{n=1}^N \left(\lambda_n^\epsilon(D)-\lambda_n^\epsilon(D^{(\alpha)})\right)^2\leq \|A\|_F^2.
\end{equation}
From this we can see that $|\lambda_i^\epsilon(D)-\lambda_i^\epsilon(D^{(\alpha)})|=O(\alpha)$ as $\alpha\to0$, since $\|A\|_F=O(\alpha)$ as $\alpha\to0$ by \Cref{lem:sizeperturb}. By a similar argument, and using \Cref{diluteness}, we have that
\begin{equation} \label{eq:dilute_eigenvalues}
    |\lambda_n(D)-\lambda_n^\epsilon(D)|=O(\epsilon^2)
    \quad\text{and}\quad
    |\lambda_n(D^{(\alpha)})-\lambda_n^\epsilon(D^{(\alpha)})|=O(\epsilon^2), \quad \text{as }\epsilon\to0.
\end{equation}
Finally, we use \Cref{thm:res} to find the resonant frequencies:
\begin{align*}
    |\omega_n(D)-\omega_n(D^{(\alpha)})|
    &=\left|\sqrt{\delta v^2\lambda_n(D)}-\sqrt{\delta v^2\lambda_n(D^{(\alpha)})}\right|+O(\delta)\\
    &\leq \sqrt{\delta v^2}\sqrt{\left| \lambda_n(D)-\lambda_n(D^{(\alpha)}) \right|}+O(\delta).\\
    &\leq \sqrt{\delta v^2}\sqrt{\left| \lambda_n(D)-\lambda_n^\epsilon(D) \right|+\left| \lambda_n^\epsilon(D)-\lambda_n^\epsilon(D^{(\alpha)}) \right|+\left| \lambda_n^\epsilon(D^{(\alpha)})-\lambda_n(D^{(\alpha)}) \right|}+O(\delta).
\end{align*}
Combining this with \eqref{eq:WielandtHoffman} and \eqref{eq:dilute_eigenvalues} gives the result.
\end{proof}

\begin{remark}
While the Wielandt-Hoffman theorem was used in \eqref{eq:WielandtHoffman}, there are a range of results that could be invoked here. For example, if $\lambda_{\text{min}}$ and $\lambda_{\text{max}}$ are the smallest and largest eigenvalues of $A$ then it holds that 
\begin{equation*}
    \lambda_n^\epsilon(D)+\lambda_{\text{min}}\leq\lambda_i^\epsilon(D^{(\alpha)})\leq \lambda_n^\epsilon(D)+\lambda_{\text{max}},
\end{equation*}
for all $n=1,\dots,N$. For a selection of results on perturbations of eigenvalues of symmetric metrices, see \cite{golub2983matrix}.
\end{remark}

\begin{figure}
    \centering
    \begin{subfigure}{0.49\linewidth}
        \includegraphics[width=\linewidth]{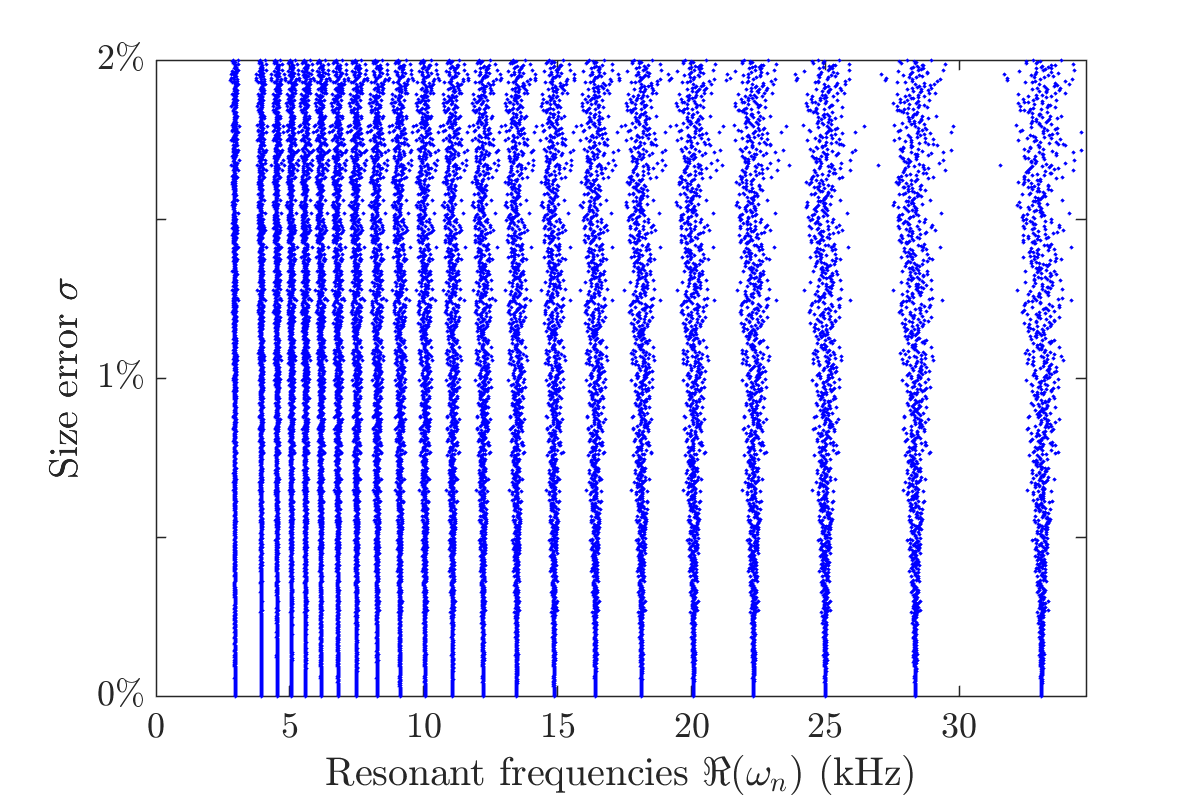}
        \caption{}
    \end{subfigure}
    \begin{subfigure}{0.49\linewidth}
        \includegraphics[width=\linewidth]{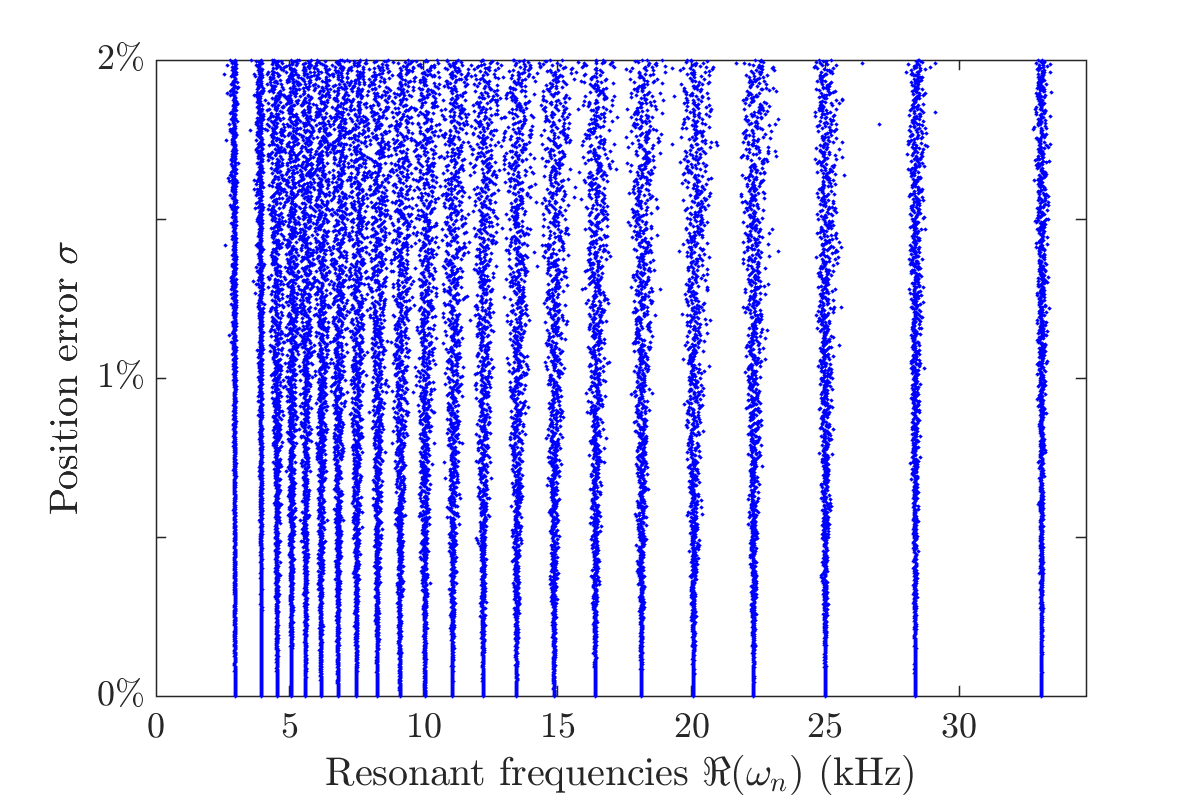}
        \caption{}
    \end{subfigure}
    \caption{The effect of random errors and imperfections on the subwavelength resonant frequencies of a cochlea-inspired rainbow sensor. (a) Random errors are added to the sizes of the resonators. (b) Random errors are added to the positions of the resonators. In both cases the errors are Gaussian with mean zero and variance $\sigma^2$.}
    \label{fig:imperfections}
\end{figure}

\subsection{Changes in position}

Let's now consider imperfections due to changes in the positions of the resonators. In particular, suppose there exist some vectors $\beta_1,\dots,\beta_N\in\mathbb{R}^3$ such that the perturbed structure is given by
\begin{equation} \label{eq:posdef}
    D^{(\beta)}=\bigcup_{i=1}^N \left( B_i+\epsilon^{-1}(z_i+\beta_i)\right).
\end{equation}
We will assume that the perturbations $\beta_1,\dots,\beta_N$ are small in the sense that there exists some parameter $\beta\in\mathbb{R}$ such that $\|\beta_i\|=O(\beta)$ as $\beta\to0$. We will proceed as in \Cref{sec:size}, by considering the dilute generalized capacitance matrix $\C^\epsilon$.

\begin{lemma} \label{lem:posperturb}
Suppose that a resonator array $D$ is deformed to give $D^{(\beta)}$, as defined in \eqref{eq:posdef}, and that the translation vectors $\beta_1,\dots,\beta_N$ satisfy $\|\beta_i\|=O(\beta)$ as $\beta\to0$ for all $i=1,\dots,N$. Then, the dilute generalized capacitance matrix associated to $D^{(\beta)}$is given by
\begin{equation*}
    \C^\epsilon(D^{(\beta)})=\C^\epsilon(D)+B(\beta),
\end{equation*}
where $B(\beta)$ is a symmetric ${N\times N}$-matrix whose Frobenius norm satisfies $\|B\|_F=O(\beta)$ as $\beta\to0$.
\end{lemma}
\begin{proof} We will make the substitution $z_i\mapsto z_i+\beta_i$ in \Cref{defn:diluteGCM}. The diagonal entries of $\C^\epsilon$ are unchanged. For the off-diagonal entries, we have that
\begin{equation*}
    \C_{ij}^\epsilon(D^{(\beta)})=
-\epsilon\frac{\mathrm{Cap}_{B_{i}} \mathrm{Cap}_{B_{j}}}{4\pi |z_{i}+\beta_i-z_{j}-\beta_j|\sqrt{|B_i||B_j|}}, \qquad i\neq j.
\end{equation*}
For small $\beta$ we can expand the denominator to give
\begin{equation*}
    \frac{1}{|z_{i}+\beta_i-z_{j}-\beta_j|}=\frac{1}{|z_{i}-z_{j}|}-(\beta_i-\beta_j)\cdot\frac{z_i-z_j}{|z_i-z_j|^3}+O(\beta^2),\qquad i\neq j,
\end{equation*}
as $\beta\to0$. This gives us that
\begin{equation*}
    \C_{ij}^\epsilon(D^{(\beta)})=\C_{ij}^\epsilon(D)
    +\epsilon(\beta_i-\beta_j)\cdot\frac{(z_i-z_j)\mathrm{Cap}_{B_{i}} \mathrm{Cap}_{B_{j}}}{4\pi |z_{i}-z_{j}|^3\sqrt{|B_i||B_j|}}+O(\beta^2), \quad i\neq j,
\end{equation*}
as $\beta\to0$.
\end{proof}

\begin{thm}
    Suppose that a resonator array $D$ is $\epsilon$-dilute in the sense of \eqref{eq:dilute} and is deformed to give $D^{(\beta)}$, as defined in \eqref{eq:posdef}, for translation vectors $\beta_1,\dots,\beta_N$ which satisfy $\|\beta_i\|=O(\beta)$ as $\beta\to0$ for all $i=1,\dots,N$. Then the resonant frequencies satisfy
    \begin{equation*}
        |\omega_n(D)-\omega_n(D^{(\beta)})|=O\left(\sqrt{\delta(\beta+\epsilon^2)}\right).\qedhere
    \end{equation*}
    as $\beta,\delta,\epsilon\to0$.
\end{thm}
\begin{proof}
From \Cref{lem:posperturb} we have that $\C^\epsilon(D^{(\beta)})=\C^\epsilon(D)+B(\beta)$ where $B$ is a symmetric $N\times N$-matrix so we can proceed as in \Cref{thm:size} to use the Wielandt-Hoffman theorem to bound $|\lambda_n^\epsilon(D)-\lambda_n^\epsilon(D^{(\beta)})|$ by $\|B\|_F$ for each $n=1,\dots,N$.  Then, approximating under the assumption that $\delta$ and $\epsilon$ are small gives the result.
\end{proof}

\subsection{Higher-order results}

Recall the formula $\omega_n=\sqrt{\delta v^2 \lambda_n}-\iu\delta\tau_n+\dots$ from \Cref{thm:res}. The formula for $\tau_n$ involves the eigenvectors $\vn$ of the generalized capacitance matrix. Assuming the material parameters are real, the $O(\delta)$ term describes the imaginary part of the resonant frequency, so it is important to understand how it is affected by imperfections in the structure.

\begin{figure}
    \begin{center}
        \includegraphics[width=0.5\linewidth]{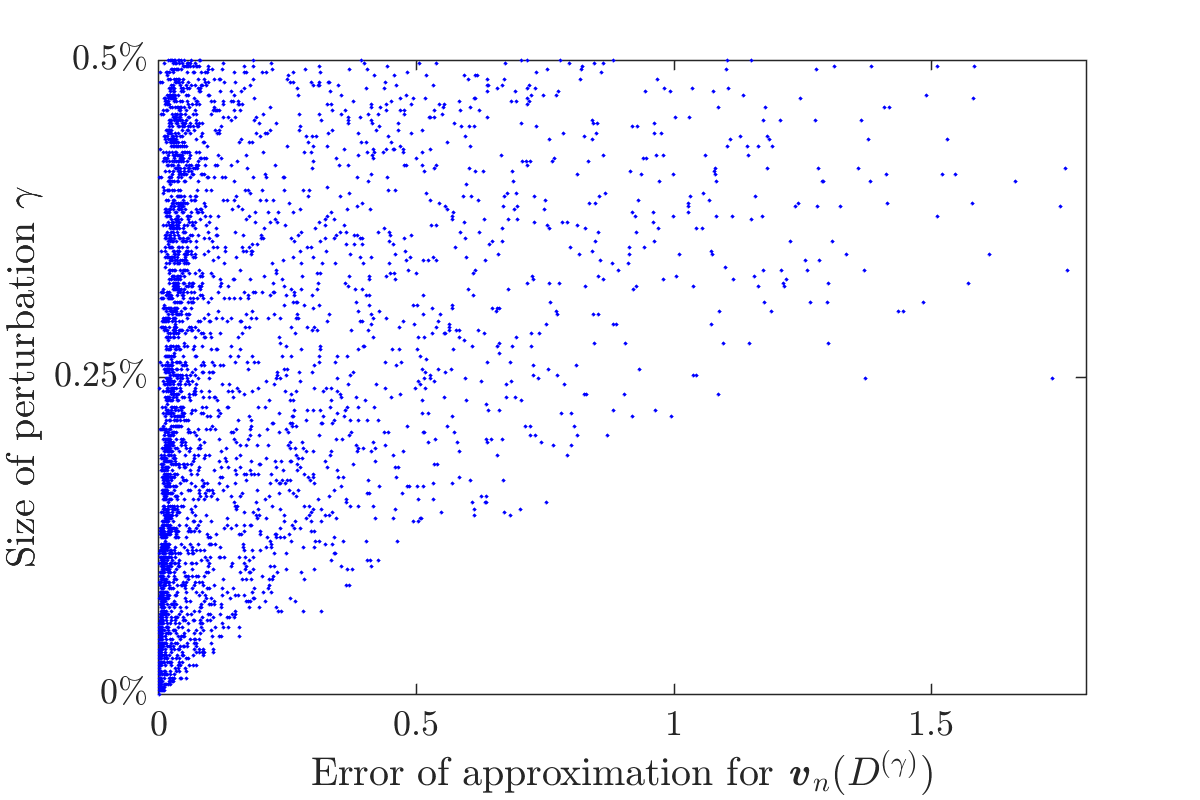}
    \end{center}
    \caption{The error of the approximation for $\vn(D^{(\gamma)})$ derived in \Cref{lem:vec_approx} is small for small perturbations $\gamma$. We repeatedly simulate randomly perturbed cochlea-inspired rainbow sensors and compare the exact value with the approximate value from \Cref{lem:vec_approx}.} \label{fig:eigenvec_perturb}
\end{figure}

\begin{lemma} \label{lem:vec_approx}
    Consider a resonator array $D$ that is such that the associated (symmetric) generalized capacitance matrix $\C(D)$ has $N$ distinct, simple eigenvalues. Suppose that a perturbation, governed by the parameter $\gamma$, is made to the structure to give $D^{(\gamma)}$ and that there is a symmetric matrix $\Gamma(\gamma)$ which is such that
    \begin{equation*}
        \C(D^{(\gamma)})=\C(D)+\Gamma(\gamma),
    \end{equation*}
    and $\|\Gamma(\gamma)\|\to0$ as $\gamma\to0$. Then, the perturbed eigenvectors can be approximated as
    \begin{equation*}
    \vn(D^{(\gamma)})\approx \vn(D) + 
    \sum_{\substack{k=1\\k\neq n}}^N \frac{\langle \Gamma(\gamma)\vn(D), \vk(D)\rangle}{(\lambda_n-\lambda_k)} \vk(D),
\end{equation*}
provided that $\gamma$ is sufficiently small.
\end{lemma}
\begin{proof}
Since $\C(D)$ is a symmetric matrix, it has an orthonormal basis of eigenvectors $\{\vn:n=1,\dots,N\}$ with associated eigenvalues $\sigma(\C(D))=\{\lambda_n: n=1,\dots,N\}$, which are assumed to be distinct. Under this assumption, we have the decomposition
\begin{equation} \label{eq:decomp}
    (\lambda I-\C(D))^{-1}x=\sum_{k=1}^N \frac{\langle x, \vk\rangle}{\lambda-\lambda_k} \vk, \quad x\in\mathbb{C}^n, \, \lambda\in\mathbb{C}\setminus\sigma(\C).
\end{equation}
From this we can see that $\|(\lambda I-\C(D))^{-1}\|\leq \text{dist}(\lambda,\sigma(\C(D)))^{-1}$. If we add a perturbation matrix $\Gamma(\gamma)$ which is such that $\|\Gamma(\gamma)\|<\text{dist}(\lambda,\sigma(\C(D)))$, then $\lambda I-\C(D^{(\gamma)})=\lambda I-\C(D)-\Gamma(\gamma)$ is invertible. Further, in this case, we can use a Neumann series to see that
\begin{equation} \label{eq:Neumann}
    (\lambda I-\C(D^{(\gamma)}))^{-1}=(\lambda I-\C(D)-\Gamma)^{-1}=(\lambda I-\C(D))^{-1} \sum_{i=0}^\infty \Gamma^i \left( (\lambda I-\C(D))^{-1} \right)^i.
\end{equation}
Substituting the decomposition \eqref{eq:decomp} and taking only the first two terms from \eqref{eq:Neumann}, we see that for a fixed $\lambda\in\mathbb{C}\setminus\sigma(\C)$ we have
\begin{equation} \label{eq:modNeumann}
    (\lambda I-\C(D^{(\gamma)}))^{-1}
    =\sum_{k=1}^N \frac{\langle \,\cdot\, , \vk\rangle}{\lambda-\lambda_k} \vk
    + \sum_{k=1}^N\sum_{j=1}^N \frac{\langle \,\cdot\, , \vj\rangle \langle \Gamma\vj, \vk\rangle}{(\lambda-\lambda_k)(\lambda-\lambda_j)} \vk+\dots,
\end{equation}
where the remainder terms are $O(\|\Gamma(\gamma)\|^2)$ as $\gamma\to0$.

Suppose we have a collection of closed curves $\{\eta_n:n=1,\dots,N\}$ which do not intersect and are such that the interior of each curve $\eta_n$ contains exactly one eigenvalue $\lambda_n$. We know that we may choose $\gamma$ to be sufficiently small that the eigenvalues of $\C(D^{(\gamma)})$ remain within the interior of these same curves. Thus, the operator $P_n:\mathbb{C}^N\to\mathbb{C}^N$, defined by
\begin{equation}
    \mathcal{P}_n=\frac{1}{2\pi\iu} \int_{\eta_n} (\lambda I-\C(D^{(\gamma)}))^{-1} \de\lambda,
\end{equation}
is the projection onto the eigenspace associated to the perturbed eigenvalue $\lambda_n(D^{(\gamma)})$. Using the expansion \eqref{eq:modNeumann}, we can calculate an approximation to the operator $\mathcal{P}_n$, given by 
\begin{equation*}
    \mathcal{P}_n\approx\langle \,\cdot\, , \vn\rangle \vn + 
    \sum_{\substack{k=1\\k\neq n}}^N \frac{\langle \,\cdot\, , \vn\rangle \langle \Gamma\vn, \vk\rangle}{(\lambda_n-\lambda_k)} \vk,
\end{equation*}
where we are assume the remainder term to be small (this is a technical issue, due to the non-uniformity of the expansion \eqref{eq:modNeumann} near to $\lambda\in\sigma(\C(D))$). Applying this approximation for the operator $\mathcal{P}_n$ to the unperturbed eigenvector $\vn$ gives the result.
\end{proof}

\Cref{lem:vec_approx} gives an approximate value for the eigenvectors of the generalized capacitance matrix when small perturbations have been made to an array of subwavelength resonators. It does not include estimates for the error, however we the accuracy of the formula has been verified by simulation. In \Cref{fig:eigenvec_perturb}, we show the norm of the difference between the formula from \Cref{lem:vec_approx} and the true eigenvector for many randomly perturbed cochlea-inspired rainbow sensors. We see that the errors are small when the size of the perturbations $\gamma$ is small.

\section{Removing resonators from the device} \label{sec:remove}

We will now consider a different class of perturbations of the rainbow sensors: the effect of removing a resonator from the array. This is shown in \Cref{fig:removal_sketch}. This is inspired by observations of the biological cochlea where in many places the receptor cells are so badly damaged that the stereocilia have been completely destroyed, as depicted in \Cref{fig:cochlea_damage}. 

\begin{figure}
    \begin{center}
        \includegraphics[width=0.6\linewidth]{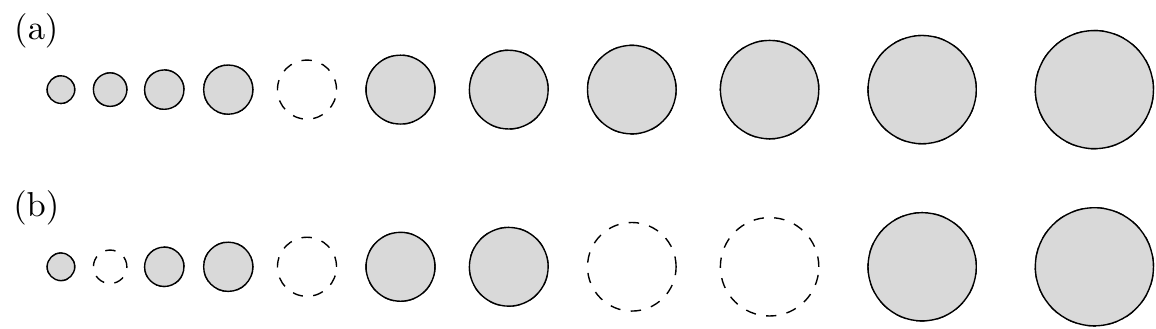}
    \end{center}
    \caption{We study the effects of removing resonators from a cochlea-inspired rainbow sensor. (a) The rainbow sensor with a single resonator removed, denoted $D^{(5)}$. (b) The rainbow sensor with multiple resonators removed, denoted $D^{(2,5,8,9)}$.  The original rainbow sensor, $D=D_1\cup\dots\cup D_{11}$, is shown in dashed.} \label{fig:removal_sketch}
\end{figure}

We introduce some notation to describe a system of resonators with one or more resonators removed. Given a resonator array $D$ we write $D^{(i)}$ to denote the same array with the $i$\textsuperscript{th} resonator removed. The resonators are labelled according to increasing volume (so, from left to right in the graded cochlea-inspired rainbow sensors depicted here, as in \Cref{fig:rainbow}). For the removal of multiple resonators we add additional subscripts. For example, in \Cref{fig:removal_sketch}(a) we show $D^{(5)}=D_1\cup\dots\cup D_4\cup D_6\cup\dots\cup D_{11}$ and in \Cref{fig:removal_sketch}(b) we show $D^{(2,5,8,9)}$, which has the 2\textsuperscript{nd}, 5\textsuperscript{th}, 8\textsuperscript{th} and 9\textsuperscript{th} resonators removed. 

The crucial result that underpins the analysis in this section is Cauchy's Interlacing Theorem, which describes the relation between a Hermitian matrix's eigenvalues and the eigenvalues of its principal submatrices. A principle submatrix is a matrix obtained by removing rows and columns (with the same indices) from a matrix.

\begin{thm}[Cauchy's Interlacing Theorem] \label{thm:interlace}
    Let $A$ be an $N \times N$ Hermitian matrix with eigenvalues $\lambda_1\leq\lambda_2\leq\dots\leq\lambda_N$. Suppose that $B$ is an $(N - 1) \times (N - 1)$ principal submatrix of $A$ with eigenvalues $\mu_1 \leq \mu_2 \leq\dots\leq \mu_{N-1}$. Then, the eigenvalues are ordered such that $\lambda_1\leq\mu_1\leq\lambda_2\leq\mu_2\leq\dots\leq\lambda_{N-1}\leq\mu_{N-1}\leq\lambda_N$.
\end{thm}
\begin{proof}
Various proof strategies exist, see \cite{golub2983matrix} or \cite{hwang2004cauchy}, for example.
\end{proof}

Thanks to Cauchy's Interlacing Theorem, we can quickly obtain a result for the eigenvalues of the generalized capacitance matrix. In order to state a result for the resonant frequencies of a resonator array, we will first introduce some asymptotic notation.

\begin{defn}
For real-valued functions $f$ and $g$, we will write that $f(\delta)\gtrsim g(\delta)$ as $\delta\to0$ if 
\begin{equation*}
    \lim_{\delta\to0} \frac{f(\delta)}{\max\{f(\delta),g(\delta)\}}=1, \qquad\text{as }\delta\to0, 
\end{equation*}
where we define the ratio to be 1 in the event that $0=f\geq g$.
\end{defn}

\begin{lemma} \label{interlacing}
    Let $D$ be a resonator array and $D^{(i)}$ be the same array with the $i$\textsuperscript{th} resonator removed. Then, if $\delta$ is sufficiently small, the resonant frequencies of the two structures interlace in the sense that
    \begin{equation*}
        \Re(\omega_j(D))\lesssim\Re(\omega_j(D^{(i)}))\lesssim\Re(\omega_{j+1}(D)) \quad\text{for all } j=1,\dots,N-1.
    \end{equation*}
\end{lemma}
\begin{proof}
Since $\C(D)$ is symmetric and real valued, we can use Cauchy's Interlacing Theorem (\Cref{thm:interlace}) to see that 
\begin{equation*}
        \lambda_j(D)\leq\lambda_j(D^{(i)})\leq\lambda_{j+1}(D) \quad\text{for all } j=1,\dots,N-1.
    \end{equation*}
Then, the result follows from the asymptotic formula in \Cref{thm:res}.
\end{proof}

\begin{figure}
    \centering
    \includegraphics[width=\linewidth,trim=0 0.5cm 0 0.8cm,clip]{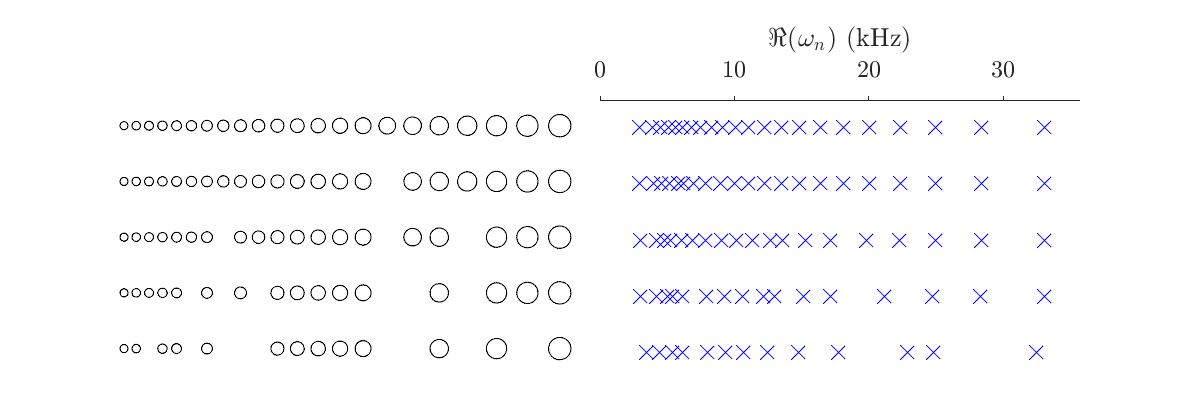}
    \caption{The subwavelength resonant frequencies of a cochlea-inspired rainbow sensor with resonators removed. Each subsequent array has additional resonators removed and its set of resonant frequencies interlaces the previous, at leading order, as predicted by \Cref{interlacing}.}
    \label{fig:removal}
\end{figure}

The subwavelength resonant frequencies of resonator arrays with an increasing number of removed resonators are shown in \Cref{fig:removal}. We see that the frequencies interlace those of the previous structure and remain distributed across the audible range.

\subsection{Stable removal from large devices} \label{sec:large}

In general, \Cref{interlacing} is useful for understanding the effect of removing a resonator but does not give stability, in the sense of the perturbation being small. However, a cochlea-inspired rainbow sensor with a large number of resonators can be designed such that the resonant frequencies are bounded, even as their number becomes very large. In this case, many of the gaps between the real parts will be small and, subsequently, so will the perturbations caused by removing a resonator. There are a variety of ways to formulate this precisely, one version is given in the following theorem.

\begin{thm} \label{thm:largearray}
Suppose that a resonator array $D$ is dilute with parameter $0<\epsilon\ll1$ in the sense that
\begin{equation*}
    D=\bigcup\limits_{j=1}^{N}(B+\epsilon^{-1}z_{j}),
\end{equation*}
where $B$ is a fixed bounded domain and $\epsilon^{-1}z_{j}$ represents the position of each resonator. In this case, the leading-order approximation of the generalized capacitance matrix is given by $\epsilon^2\C^\epsilon$ as $\epsilon\to0$ (where $\C^\epsilon$ was defined in \Cref{defn:diluteGCM}). Further, there exists a constant $c\in\mathbb{R}$, which does not depend on $N$ or $\epsilon$, such that if $\epsilon=\frac{c}{N}$, then all the eigenvalues $\{\lambda_{j}\}$ of $\epsilon^2\C^\epsilon$ are such that
\begin{align}
    0<\lambda_{j}<\frac{2|\mathrm{Cap}_{B}|}{|B|}.
\end{align}
\end{thm}
\begin{proof}
In this case, it is easy to check that the leading-order approximation of the generalized capacitance matrix is given by
\begin{align}
\epsilon^2\C^\epsilon_{ij}= \begin{cases}
\frac{\mathrm{Cap}_{B}}{|B|}, & i=j, \\
-\frac{ \epsilon \mathrm{Cap}_{B}^{2}}{4\pi|B| |z_{i}-z_{j}|}, & i\neq j,\\
\end{cases}
\end{align}
as $\epsilon\to0$. By the Gershgorin circle theorem we know that the eigenvalues $\{\lambda_{j}: j=1,\dots,N\}$ must be such that
\begin{align}
    \left|\lambda_{j}-\frac{\mathrm{Cap}_{B}}{|B|} \right|\leq \frac{ \epsilon\mathrm{Cap}_{B}^{2}}{4\pi|B|} \sum_{i\neq j} \frac{1}{|z_{i}-z_{j}|}, \quad j=1,\dots,N.
\end{align}
Now, we have that 
\begin{equation*}
    \frac{ \epsilon\mathrm{Cap}_{B}}{4\pi} \sum_{i\neq j} \frac{1}{|z_{i}-z_{j}|}
    \leq \epsilon(N-1)\frac{\mathrm{Cap}_{B}}{4\pi} \sup_{i\neq j} |z_{i}-z_{j}|^{-1},
\end{equation*}
which we can choose to be less than 1 by selecting $c=\epsilon N$ appropriately. In which case, we have that the eigenvalues $\{\lambda_{j}: j=1,\dots,N\}$ satisfy 
\begin{equation*}
    \left|\lambda_{j}-\frac{\mathrm{Cap}_{B}}{|B|} \right|\leq \frac{\mathrm{Cap}_{B}}{|B|}, \quad j=1,\dots,N. \qedhere
\end{equation*}
\end{proof}

It is important to note that \Cref{thm:largearray} merely shows that the real parts of the resonant frequencies will be bounded, as the number of resonators becomes large. It does not guarantee that they are evenly spaced or that the gaps between any particular adjacent resonant frequencies are small. For example, see \Cref{fig:largearray}, where the subwavelength resonant frequencies for increasingly large arrays, dimensioned according to \Cref{thm:largearray}, are shown. We see that the frequencies become very dense in part of the range but remain sparser at higher frequencies.

\begin{figure}
    \centering
    \includegraphics[width=0.6\linewidth,trim=0 1.7cm 0 0, clip]{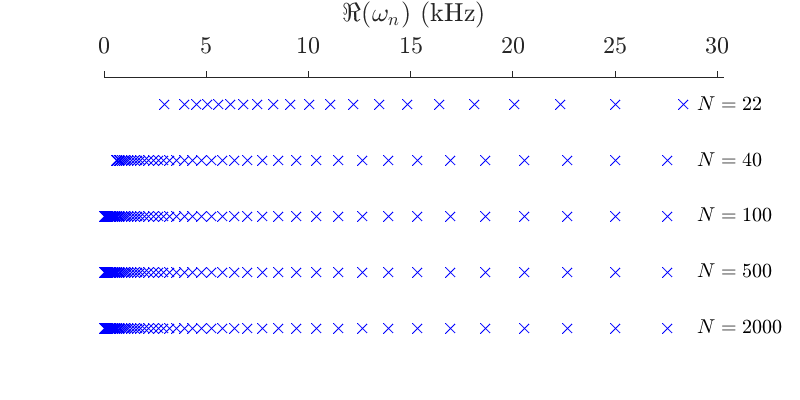}
    \caption{Large cochlea-inspired rainbow sensors can be designed such that the subwavelength resonant frequencies are bounded. Here, we simulate successively larger arrays, according to the dilute regime defined in \Cref{thm:largearray}.}
    \label{fig:largearray}
\end{figure}

\section{Implications for signal processing} \label{sec:signals}

The aim of the cochlea-like rainbow sensor studied in this work is to replicate the ability of the cochlea to filter sounds. There is also a large community of researchers developing signal processing algorithms with the same aim: to replicate the abilities of the human auditory system. Since we have precise analytic methods to describe how the array scatters an incoming field, we can draw comparisons between the cochlea-inspired rainbow sensor studied here and biomimetic signal transforms. This is explored in detail in \cite{ammari2020biomimetic}. In particular, given a formula for the field that is scattered by the cochlea-inspired rainbow sensor, we can deduce the corresponding signal transform. In this section, we explore how this signal transform is affected by the introduction of errors and imperfections. 

\subsection{A biomimetic signal transform} \label{sec:algorithm}

We briefly recall from \cite{ammari2020biomimetic} how a biomimetic signal transform can be deduced from a cochlea-inspired rainbow sensor. In response to an incoming wave $\uin$, the solution to the Helmtolz problem \eqref{eq:helmholtz_equation} is given, for $x\in\outside$, as
\begin{equation} \label{eq:frequencydecomp}
u(x)-\uin(x) = \sum_{n=1}^N q_n \S_D^k[\psi_n](x) - \S_D^k\left[\S_D^{-1}[\uin]\right](x) + O(\omega),
\end{equation} 
as $\omega\to0$, for constants $q_n$ which satisfy
\begin{equation}\label{eq:eval_C}
	\left({\omega^2}I-{v_b^2\delta}\,\C\right)\begin{pmatrix}q_1\\ \vdots\\q_N\end{pmatrix}
	=
	{v_b^2\delta}\begin{pmatrix} \frac{1}{|D_1|} \int_{\D_1}\S_D^{-1}[\uin]\de\sigma \\ \vdots\\
	\frac{1}{|D_N|}\int_{\D_N}\S_D^{-1}[\uin]\de\sigma \end{pmatrix}+O(\delta \omega+\omega^3),
\end{equation}
as $\omega,\delta\to0$. Suppose that the incoming wave is a plane wave and can be written in terms of some real-valued function $s$ as
\begin{equation}
    \uin(x,\omega)=\int_{-\infty}^{\infty} s(x_1/v-t)e^{\iu\omega t} \de t.
\end{equation}
Assuming that we are in an appropriate low-frequency regime, such that the remainder terms remain small, we can apply a Fourier transform to \eqref{eq:frequencydecomp} to see that the scattered pressure field $p(x,t)$ is given by
\begin{equation*} 
p(x,t)= \sum_{n=1}^N a_n[s](t) u_n(x) + ...,
\end{equation*}
where the remainder term is $O(\delta)$ and the coefficients are given by
\begin{equation} \label{eq:andefn}
    a_n[s](t)=\left( s*h[\omega_n] \right)(t), \qquad n=1,\dots,N,
\end{equation}
for kernels defined as
\begin{equation} \label{eq:hdefn}
h[\omega_n](t)=
\begin{cases}
0, & t<0, \\
c_n e^{\Im(\omega_n)t} \sin(\Re(\omega_n)t), & t\geq0,
\end{cases} \qquad n=1,\dots,N,
\end{equation}
for some real-valued constants $c_n$. See \cite{ammari2020biomimetic} for details. Thus, the deduced signal transform is: given a signal $s$, compute the $N$ time-varying outputs $a_n[s]$, defined by \eqref{eq:andefn}. 

\subsection{Stability to errors}

We wish to show that the signal transform $s\mapsto a_n[s]:=s*h[\omega_n]$ is robust with respect to errors and imperfections in the design of the underlying cochlea-inspired rainbow sensor.

\begin{thm} \label{thm:stability}
Given two complex numbers $\omold$ and $\omnew$ with negative imaginary parts, it holds that
\begin{align*}
    \left\|s*h[\omold]-s*h[\omnew]\right\|_{L^{\infty}(\mathbb{R})} &\leq \|h[\omold]-h[\omnew]\|_{L^{\infty}(\mathbb{R}))}\|s\|_{L^{1}(\mathbb{R})},
\end{align*}
for all $s \in L^{1}(\mathbb{R})$. 
\end{thm}
\begin{proof} This is a standard argument for bounding convolutions:
\begin{align*}
    \|s*h[\omold]-s*h[\omnew]\|_{L^{\infty}(\mathbb{R})}
      &\leq \sup_{x\in \mathbb{R}}  \int_{\mathbb{R}}|s(x-y)|h[\omold](y)-h[\omnew](y)| \de y \\ 
    &\leq \|h_{n}^{old}-h_{n}^{new}\|_{L^{\infty}(\mathbb{R})} \sup_{x\in \mathbb{R}}\int_{\mathbb{R}}|s(x-y)| \de y \\
    &=\|h_{n}^{old}-h_{n}^{new}\|_{L^{\infty}(\mathbb{R})}\|s\|_{L^{1}(\mathbb{R})}. \qedhere
\end{align*}
\end{proof}

\begin{remark}
If $s$ is compactly supported, then we can reframe  \Cref{thm:stability} in terms of $\|\cdot\|_{L^p(\mathbb{R})}$ for any $1\leq p\leq\infty$, using H\"older's inequality.
\end{remark}

\begin{cor} \label{cor:stability}
    Let $c>0$ and suppose we have two complex numbers $\omold$ and $\omnew$ whose imaginary parts satisfy $\Im(\omold),\Im(\omold)\leq-c$. Then, it holds that
    \begin{align*}
    \left\|s*h[\omold]-s*h[\omnew]\right\|_{L^{\infty}(\mathbb{R})} 
    \leq \frac{\sqrt{2}}{ce}\left|\omold-\omnew\right|\|s\|_{L^{1}(\mathbb{R})},
    \end{align*}
    for all $s \in L^{1}(\mathbb{R})$.
\end{cor}
\begin{proof}
We begin with the observation that
\begin{align*}
    &\left|h[\omold](t)-h[\omnew](t)\right|\\&\qquad\qquad=\left|\left(e^{\Im(\omold)t}-e^{\Im(\omnew)t}\right)\sin(\Re(\omold)t) +e^{\Im(\omnew)t} \left(\sin(\Re(\omold)t)-\sin(\Re(\omnew)t) \right)\right|,
\end{align*}
for $t>0$. Then, we have that
\begin{align*}
    \left|\left(e^{\Im(\omold)t}-e^{\Im(\omnew)t}\right)\sin(\Re(\omold)t)\right|\leq 
    \left|e^{\Im(\omold)t}-e^{\Im(\omnew)t}\right|\leq \frac{1}{ce}|\Im(\omold)-\Im(\omnew)|,
\end{align*}
for $t>0$, where we have used the fact that $\sup_{t>0}\sup_{\omega<-c}|te^{\omega t}|=\frac{1}{ce}$. Similarly, we have that
\begin{equation*}
    \left|e^{\Im(\omnew)t} \left(\sin(\Re(\omold)t)-\sin(\Re(\omnew)t) \right)\right|\leq\frac{1}{ce}|\Re(\omold)-\Re(\omnew)|.
\end{equation*}
for $t>0$, where we have used the fact that $\sup_{t>0}\sup_{\omega<-c}|te^{\omega t}\cos(at)|\leq\frac{1}{ce}$ for any $a\in\mathbb{R}$. Putting this together, we have that
\begin{align*}
    \left\|s*h[\omold]-s*h[\omnew]\right\|_{L^{\infty}(\mathbb{R})} 
    \leq \frac{1}{ce}\Big(\left|\Im(\omold)-\Im(\omnew)\right|+\left|\Re(\omold)-\Re(\omnew)\right|\Big)\|s\|_{L^{1}(\mathbb{R})},
\end{align*}
from which we arrive at the result, using the inequality $|a|+|b|\leq\sqrt{2(a^2+b^2)}$.
\end{proof}

While \Cref{thm:stability} is the standard stability result for convolutional signal processing algorithms, \Cref{cor:stability} is most revealing here. It shows that the outputs of the induced biomimetic signal transform (defined by \eqref{eq:andefn} here) are stable with respect to changes in the resonant frequencies of the physical device. From Sections~\ref{sec:imperfections} and \ref{sec:remove}, we know that the resonant frequencies of the cochlea-inspired rainbow sensor are robust with respect to a variety of errors and imperfections (particularly in large resonator arrays), meaning that the biomimetic signal transform inherits this robustness. 

To test the robustness for small arrays with removed resonators, \Cref{fig:freqsupport} shows the frequency support of the filter array used in the biomimetic signal transform in the case of successively removed resonators (the same sequence of structures was simulated in \Cref{fig:removal}). In this small array (of 22 resonators, initially) we see that gaps emerge when multiple resonators are removed, corresponding to hearing loss at frequencies within these gaps. It is interesting to note that the gaps emerge at higher frequencies. This was observed in many simulations and is commensurate with the wider spacing of frequencies at the upper end of the audible range (see \Cref{fig:largearray}, for example) and, interestingly, is consistent with the observation that human hearing loss initially occurs at high frequencies in most people \cite{wu2020age}.

\begin{figure}
    \centering
    \includegraphics[width=\linewidth,trim=0 1.5cm 0 3.3cm, clip]{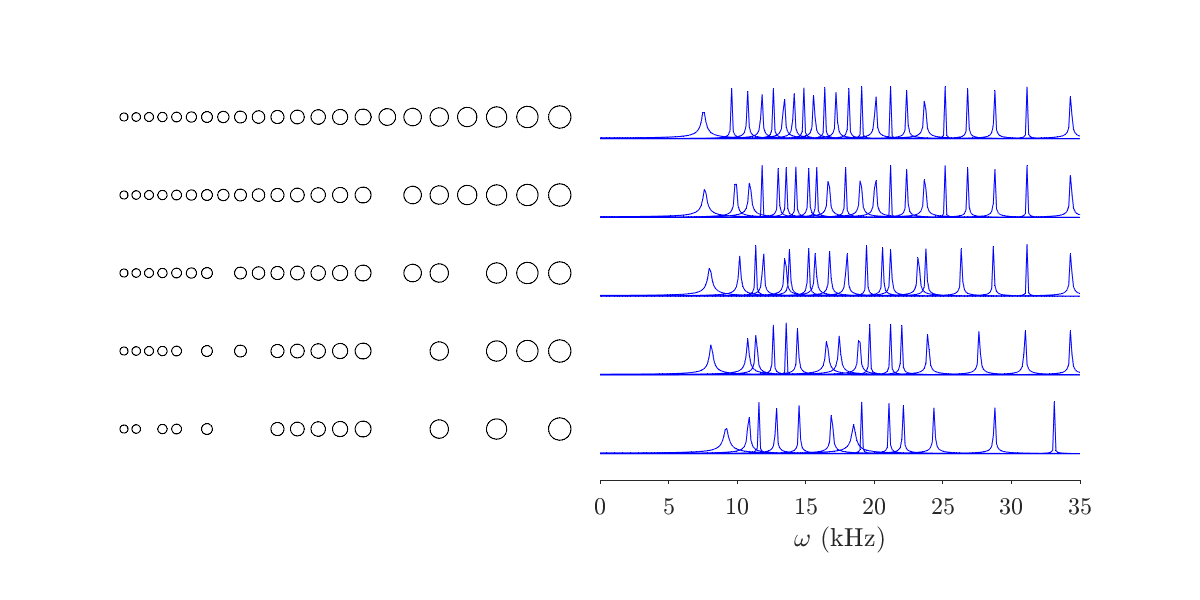}
    \caption{The frequency supports of the filter kernels $h[\omega_n]$ induced by a cochlea-inspired rainbow sensor. Each subsequent array has additional resonators removed and for each array we plot the Fourier transform of $h[\omega_n]$, $n=1,\dots,N$, normalized in $L^2(\mathbb{R})$.}
    \label{fig:freqsupport}
\end{figure}

\section{Concluding remarks}

The formulas derived in this work show that a cochlea-inspired rainbow sensor is robust with respect to small perturbations in the position and size of the constituent resonators. The effect of removing resonators was also described; it was shown that the change in the subwavelength resonant frequencies is always bounded (via an interlacing theorem) and can be small in the case of sufficiently large arrays. The implication of this analysis for related signal transforms were also studied, and it was shown that stability properties are inherited from the underlying resonant frequencies. The implications for the the corresponding biomimietic signal transform were also studied, and it was shown that this inherits the robustness of the device's resonant frequencies.

The analysis in this work (\Cref{sec:large}, in particular) suggests a possible mechanism through which a sufficiently large structure could be robust to (surprisingly) large perturbations. However, the extent to which this truly replicates the remarkable robustness of the cochlea is unclear. While the mechanisms which underpin the function of cochlea-inspired rainbow sensors (which are locally resonant graded metamaterials) and biological cochleae (which have a graded membrane with receptor cells on the surface) are quite different, there is scope for further insight to be traded between the two communities. For example, there has recently been new insight into the role of topological protection in rainbow sensors \cite{chaplain2020topological, chaplain2020delineating} and in signal processing devices \cite{zangeneh2019topological}.

\section*{Acknowledgements}
The authors would like to thank Habib Ammari for his insight and support. This work was supported by the Seminar for Applied Mathematics at ETH Zurich, where both authors were formerly based. The work of BD was also supported by the H2020 FETOpen project BOHEME under grant agreement No.~863179. The authors are also grateful to Elizabeth M Keithley for providing the micrographs in \Cref{fig:cochlea_damage}.

\section*{Data accessibility}

The code used in this study is available at \hyperlink{https://doi.org/10.5281/zenodo.5541152}{https://doi.org/10.5281/zenodo.5541152}.

\bibliographystyle{abbrv}
\bibliography{robust}{}
\end{document}